\documentclass[11pt, oneside]{article}   	
\usepackage{amsmath, amsthm, amsfonts, amssymb}
\usepackage[colorlinks,citecolor=blue,urlcolor=blue]{hyperref}
\newcommand\Ex{{\mathbb E}}

\newcommand\Prob{{\mathbb P}}

\newcommand\Normal{{\mathcal N}}

\newcommand\cA{{\mathcal A}}

\newcommand\cC{{\mathcal C}}

\newcommand\cP{{\mathcal P}}

\newcommand\N{{\mathbb N}}
\newcommand\Z{{\mathbb Z}}
\newcommand\R{{\mathbb R}}

\newcommand\fX{{\mathfrak X}}

\newcommand\dto{\overset{d}{\to }}

\newcommand\one{{\bf 1}}

\DeclareMathOperator{\Var}{Var}

\DeclareMathOperator{\interior}{int}

\newtheorem{theorem}{Theorem}[section]
\newtheorem{corollary}[theorem]{Corollary}
\newtheorem{lemma}[theorem]{Lemma}

\theoremstyle{definition}

\theoremstyle{remark}
\newtheorem{remark}[theorem]{Remark}

\title{On central limit theorems in stochastic geometry}
\author{Khanh Duy Trinh
\footnote{Research Alliance Center for Mathematical Sciences, Tohoku University, Japan. 
Email: trinh.khanh.duy.a3@tohoku.ac.jp} 
}

\begin{document}
\maketitle

\begin{abstract}
We establish central limit theorems for general functionals on binomial point processes and their  Poissonized version. As an application, a central limit theorem for Betti numbers of random geometric complexes in the thermodynamic regime is derived. 

\medskip

		
\medskip
	

\end{abstract}

\section{Introduction}

The paper introduces a new approach to establish central limit theorems (CLT) for  functionals on binomial point processes and Poisson point processes. CLTs in this setting may be found in \cite{PY-2001} for general functionals, and in \cite{BY-2005, PY-2005} for functionals of a specific form. However, the work in \cite{PY-2001} only deals with binomial point processes having uniform distribution and homogeneous Poisson point processes. We are going to remove such restrictions in this paper.

Binomial point processes considered here are $\fX_n = \{X_1, \dots, X_n\}$, where $\{X_i\}_{i = 1}^\infty$ is an i.i.d.~(independent identically distributed) sequence of $\R^d$-valued random variables having probability density function $f$. The function $f$ is assumed to be bounded and to have compact support. Associated with $\{\fX_n\}$ is the Poissonized version $\cP_n = \{X_1, \dots, X_{N_n}\}$ which becomes a Poisson point process with intensity function $nf$. Here  the random variable $N_n$ has Poisson distribution with parameter $n$ and is independent of $\{X_i\}$. By a functional, it means a real-valued measurable function $H$ defined on all finite subsets in $\R^d$. We will study CLTs for $H(n^{1/d}\fX_n)$ and $H(n^{1/d}\cP_n)$ as $n$ tends to infinity.

Let us first introduce the result in \cite{PY-2001}. Assume that $X_i$ is uniformly distributed on some bounded set $S$, or equivalently $f(x) \equiv \lambda$ on $S$. The support $S$ may need some technical assumption. In this case, the point process $n^{1/d}\cP_n$ has the same distribution with the restriction  on $n^{1/d}S$ of a homogeneous Poisson point process $\cP(\lambda)$ with intensity $\lambda$. Then a CLT holds for $H(n^{1/d}\cP_n)$, that is, $n^{-1/2}(H(n^{1/d}\cP_n) - \Ex[H(n^{1/d}\cP_n)])$ converges in distribution to a Gaussian distribution with mean $0$ and variance $\sigma^2 \ge 0$, provided that the functional $H$ is weakly stabilizing and satisfies a bounded moment condition. Here the concept of stabilization is defined via the add one cost function associated with $H$, $D_0(\cdot) = H(\cdot \cup \{0\}) - H(\cdot)$, which measures the increment of $H$ by adding a point at the origin. (The precise definition will be given in Section 3.2.) This approach is based on the martingale difference central limit theorem as one may expect due to a spatial independence property of Poisson point processes. A CLT for $H(n^{1/d} \fX_n)$ is then derived by a de-Poissonization technique in which a stronger condition, called strong stabilization, is required. Roughly speaking, $H$ is strongly stabilizing if the value of $D_0$ on $\cP(\lambda)$ does not change when  adding or removing points far from the origin. Although some techniques had been developed in \cite{Kesten-Lee-1996, Lee-1997, Lee-1999}, the paper \cite{PY-2001} is the first one successfully dealing with  general functionals. Since then,  it has found many applications.

For the non-uniform distributions case, that martingale-based approach has been shown to work for some specific functionals (eg.~the component count in geometric graph \cite[Section~13.7]{Penrose-book} and functionals related to Euclidean minimal spanning trees \cite{Lee-1999}). To the best knowledge of the author, there is no general result like \cite{PY-2001} yet. In this paper, we develop a new fundamental approach to derive CLTs for functionals which are assumed to be strongly stabilizing on $\cP(\lambda)$ for all $0 \le \lambda  \le \sup f(x)$. (Some additional bounded moments conditions are needed.) Note that we impose the strong stabilization on homogeneous Poisson point processes only. This condition is very mild in the sense that it is also a necessary condition for  the well-established de-Poissonization technique in \cite[Section~2.5]{Penrose-book}.

It is worth mentioning another direction in the study of the limiting behavior of $H(n^{1/d} \fX_n)$ and $H(n^{1/d}\cP_n)$. In this direction, assume that the functional $H$ can be expressed in the following form 
\[
	H(\fX) = \sum_{x \in \fX} \xi(x; \fX),	\quad \fX\subset \R^d\text{: finite subset},
\] 
where $\xi(x; \fX)$ is a local (or stabilizing) function. (The stabilization of $\xi(x; \fX)$ has the same meaning with the strong stabilization of $D_0$.) Under some more conditions on the tail of stabilization radii, laws of large numbers and central limit theorems have been established \cite{BY-2005, Penrose-2007, PY-2003, PY-2005}. An explicit expression for the limiting variance and a rate of convergence in CLTs have been also known.  We need not to compare those results with ours because the scope is different.

The paper is organized as follows. Section 2 introduces some probabilistic ingredients. CLTs for homogeneous Poisson point processes, non-homogeneous Poisson point processes and binomial point processes are established in turn in Section 3. A partial result on CLT for Betti numbers in the thermodynamic regime, as an application of the general theory, is discussed in Section 4.

\section{Probabilistic ingredients}
This section introduces several useful results needed in this paper.
\subsection{CLT for triangular arrays}
The following is an easy consequence of Lyapunov's central limit theorem.
\begin{theorem}\label{thm:CLT-Lyapunov}
For each $n$, let $\{\xi_{n, i}\}_{i = 1}^{\ell_n}$ be a sequence of independent real random variables. Here we require that $\ell_n \le c n$ for some constant $c>0$. Assume that
\begin{itemize}
\item[\rm(i)] $\lim_{n \to \infty} \frac{1}{n} \sum_{i = 1}^{\ell_n} \Var[\xi_{n, i}] = \sigma^2 \in [0, \infty);$
\item[\rm(ii)] for some $\delta > 0$,
$
	\sup_n\sup_{i} \Ex[|\xi_{n, i}|^{2+\delta}] < \infty.
$ 
\end{itemize}
Then 
\[
	\frac{1}{\sqrt{n}} \sum_{i = 1}^{\ell_n} \Big( \xi_{n, i} - \Ex[\xi_{n, i}]\Big) \dto \Normal(0, \sigma^2) \text{ as } n \to \infty.
\]
Here `$\dto$' denotes the convergence in distribution, and $\Normal(0, \sigma^2)$ denotes the Gaussian distribution with mean zero and variance $\sigma^2$.
\end{theorem}

We will need the following result which may be found somewhere in literature.  
\begin{lemma}\label{lem:CLT-triangle}
Let $\{Y_n\}_{n=1}^\infty$ and $\{X_{n,k}\}_{n,k=1}^\infty$ be mean zero real random variables. Assume that
\begin{itemize}
	\item[\rm(i)] for each $k$, as $n \to \infty$,
		$	X_{n,k} \dto \Normal(0, \sigma_k^2),$ and $\Var[X_{n, k}] \to \sigma_k^2$;
	\item[\rm(ii)] 
		$
			\lim_{k \to \infty} \limsup_{n \to \infty} \Var[X_{n,k} - Y_n] =0.
		$
\end{itemize}
Then the limit $\sigma^2 = \lim_{k \to \infty} \sigma_k^2$ exists, and as $n \to \infty$,
\[
 	Y_n \dto \Normal(0, \sigma^2), \quad \Var[Y_n] \to \sigma^2.
\]
\end{lemma}
\begin{proof}
It follows from the triangular inequality that 
\begin{align*}
	\Var[X_{n, k}]^{1/2} - \Var[X_{n, k} - Y_n]^{1/2} &\le \Var[Y_n]^{1/2} \\
	&\qquad\le \Var[X_{n, k}]^{1/2} + \Var[X_{n, k} - Y_n]^{1/2}.
\end{align*}
By letting $n \to \infty$ first, and then let $k \to \infty$ in the above inequalities, we see that 
\[
	\lim_{n \to \infty} \Var[Y_n] =  \lim_{k \to \infty} \sigma_k^2 =: \sigma^2. 
\]

Let $t\in \R$ be fixed. By the assumption (i), for each $k$,
\[
	\lim_{n \to \infty} \Ex[e^{it X_{n, k}}] = e^{-\sigma_k^2 t^2/2}.
\]
Similarly as above, it follows from the inequality 
\[
	|\Ex[e^{it X_{n, k}}] - \Ex[e^{it Y_n}]| \le |t|\Ex[|X_{n, k} - Y_n|] \le |t|\Var[X_{n, k} - Y_n]^{1/2},
\]
that 
\[
	\lim_{n \to \infty} \Ex[e^{it Y_n}] = \lim_{k \to \infty} \lim_{n \to \infty}  \Ex[e^{it X_{n, k}}] = e^{-\sigma^2 t^2/2}.
\]
Therefore $Y_n \dto \Normal(0, \sigma^2)$ as desired. The proof is complete.
\end{proof}

\subsection{Poisson point processes}
Let $f(x) \ge 0$ be a locally integrable function on $\R^d$. A Poisson point process with intensity function $f$ is a point process $\cP$ on $\R^d$ which satisfies the following conditions 
\begin{itemize}
	\item[(i)] for any bounded Borel set $A$, the number of points inside $A$, denoted by $\cP(A)$, has Poisson distribution with parameter $(\int_A f(x) dx)$;
	
	\item[(ii)] for disjoint Borel sets $A_1, \dots, A_k$, the random variables $\cP(A_1), \dots, \cP(A_k)$ are independent.
\end{itemize}	

A Poisson point process with the intensity function $f(x)$ identically equal to a constant $\lambda \ge 0$ is called a homogeneous Poisson point process with density $\lambda$.

We need the following result on the convergence of functional on Poisson point processes. Recall that a functional $H$ is a real-valued measurable function defined on all finite subsets in $\R^d$. 
\begin{lemma}\label{lem:L1-convergence}
Let $\{f_n\}_{n = 1}^\infty$ and $f$ be non-negative integrable functions defined on a bounded Borel set $W$. Assume that the sequence $\{f_n\}$ converges to $f$ in $L^1(W)$, that is, $\int_W |f_n(x) - f(x)|dx \to 0$ as $n \to \infty$. Then for any functional $H$, 
\[
	H(\cP(f_n)) \dto H(\cP(f)) \text{ as } n \to \infty. 
\]
Here $\cP(f_n)$ (resp.~$\cP(f)$) denotes a Poisson point process with intensity function $f_n$ (resp.~$f$).
\end{lemma}
\begin{proof}
	We use the following coupling. Let $\Phi$ be a homogeneous Poisson point process with density $1$ on $W \times [0, \infty)$. Let 
	\begin{align*}
		A_n &= \{(x, t) \in W \times [0, \infty) : t \le f_n(x)\},\\  
		A &= \{(x, t) \in W \times [0, \infty) : t\le f(x)\}.
	\end{align*}
Let $\cP_n$ (resp.~$\cP$) be the projection of the point process $\Phi|_{A_n}$ (resp.~$\Phi|_{A}$) onto $W$. Then $\cP_n$ (resp.~$\cP$)  becomes a Poisson point process with intensity function $f_n$ (resp.~$f$). 

Let $B_n = \{(x, t) \in W \times [0, \infty) : f(x)\wedge f_n(x) < t \le f(x) \vee f_n(x)\}$. Then $\cP_n \equiv \cP$, if and only if there is no point of $\Phi$ on $B_n$. Thus 
\[
	\Prob(\cP_n \equiv \cP) = \Prob(\Phi(B_n) = 0) =  \exp(-\int_W|f_n(x) - f(x)|dx ).
\]
Consequently, as $n \to \infty$, 
\[
	\Prob(H(\cP_n) = H(\cP)) \ge \Prob(\cP_n \equiv \cP) = \exp(-\int_W|f_n(x) - f(x)|dx) \to 1.
\]
It follows that on this realization, $H(\cP_n)$ converges in probability to $H(\cP)$. Therefore, $H(\cP(f_n))$ converges in distribution to $H(\cP(f))$. The proof is complete. 
\end{proof}

The functional $H$ is said to be \emph{translation-invariant} if $H(y+\fX) = H(\fX)$ for all finite subsets $\fX\subset \R^d$ and all $y \in \R^d$, where $y + \fX = \{y + x : x \in \fX\}$. For translation-invariant functional, Poisson point processes do not need to be defined on the same region. Consequently, we have:
\begin{corollary}\label{cor:L1-convergence}
	Let $H$ be a translation-invariant functional. Let $W \subset \R^d$ be a bounded Borel set. Assume that $\int_{W_n}|f_n(x) - \lambda| dx \to 0$ as $n\to \infty$, where $\{f_n\}$ are non-negative functions defined on $W_n = y_n + W$, and $\lambda \ge 0$ is a constant. Then 
	\[
		H(\cP(f_n)) \dto H(\cP(\lambda)|_W) \text{ as } n \to \infty.
	\]
Here $\cP(\lambda)|_W$ denotes the restriction on $W$ of a homogeneous Poisson point process $\cP(\lambda)$ with density $\lambda$.  

Assume further that for some $\delta > 0$,
	$
		\sup_n \Ex[|H(\cP(f_n))|^{2+ \delta}] < \infty.
	$
Then as $n \to \infty$,
\[
\Ex[H(\cP(f_n))] \to \Ex[H(\cP(\lambda)|_W)], \quad \Var[H(\cP(f_n))] \to \Var[H(\cP(\lambda)|_W)].
\]
\end{corollary}

\begin{proof}
Since the functional $H$ is translation-invariant, the first statement follows directly from the previous lemma. The second statement is a standard result in probability theory, (for example, see the corollary following Theorem 25.12 in \cite{Billingsley}). 
\end{proof}

Next, we introduce  the so-called Poincar\'e inequality for the variance of Poisson functional, an essential tool in this paper. Let $\cP$ be a Poisson point process with intensity function $f$. Assume that $\int f(x) dx < \infty$. Then almost surely, $\cP$ has finitely many points. For a functional $H$, define an add one cost function as
\[
	D_x(\fX) = H(\fX \cup \{x\}) - H(\fX).
\]
Then \cite[Eq.~1.8]{Last-Penrose-2011}
\begin{equation}\label{Poincare}
	\Var[H(\cP)] \le \Ex\bigg[ \int |D_x(\cP)|^2 f(x) dx \bigg] = \int \Ex[|D_x(\cP)|^2] f(x) dx.
\end{equation}

\section{Central limit theorems}

\subsection{Homogeneous Poisson point processes}
From now on, assume that the functional $H$ is translation-invariant. Let $\cP(f)$ (resp.~$\cP(\lambda)$) denote a Poisson point process with intensity function $f$ (resp.~homogeneous Poisson point process with density $\lambda$).

The functional $H$ is said to be \emph{weakly stabilizing} on $\cP(\lambda)$ if there is  a (finite) random variable $\Delta(\lambda)$ such that 
\[
	D_0 (\cP(\lambda)|_{V_n}) \to \Delta(\lambda), \text{almost surely},
\]
for any sequence $\{V_n \ni 0\}_{n = 1}^\infty$ of cubes which tends to $\R^d$ as $n \to \infty$. Here a cube means a subset in $\R^d$ of the form $y + [0, a)^d$.

\begin{theorem}\label{thm:homo}
	Assume that the functional $H$ is weakly stabilizing on $\cP(\lambda)$. Assume further that for some $p > 2$,
			\begin{equation}\label{p-moment}
				\sup_{0 \in W\text{:cube}}\Ex[|D_0 (\cP(\lambda)|_{W})|^p] <\infty.
			\end{equation}
	Then as $n \to \infty$,
	\[
		\frac{H(\cP_n(\lambda)) - \Ex[H(\cP_n(\lambda))]}{\sqrt{n}} \dto \Normal(0, \hat \sigma^2(\lambda)), \quad  \frac{\Var[H(\cP_n(\lambda))]}{n} \to \hat \sigma^2(\lambda).
	\]
	Here $\cP_n(\lambda) = \cP(\lambda)|_{[0, n^{1/d})^d}$, and $n$ is not necessary an integer number.
\end{theorem}
This theorem (with $p = 4$) is a special case of Theorem 3.1 in \cite{PY-2001} in which the restriction of $\cP(\lambda)$ on a general sequence of subsets $\{B_n\}$ was considered. Thus the weak stabilization and the moment condition~\eqref{p-moment} should be defined in terms of $\{B_n\}$.

\begin{proof} For $L > 0$, and for each $n$, divide the cube $K_n := [0, n^{1/d})^d$ according to the lattice $L^{1/d} \Z^d$ and let $\{W_i\}$ be the lattice cubes which are  entirely contained in $K_n$. Let 
\[
	\quad  X_{n, L} = \frac{1}{\sqrt n} \sum_{i} \Big(H(\cP_{W_i}) - \Ex[H(\cP_{W_i})] \Big) = \frac{\sqrt{\ell_n}}{\sqrt{n}} \frac{1}{\sqrt{\ell_n}}  \sum_{i} (\cdots ) ,
\]
$\ell_n$ being the number of cubes $\{W_i\}$. Here for simplicity, we remove $\lambda$ in formulae. Then $X_{n,L}$ is a (scaled) sum of i.i.d.~mean zero random variables. Note that the variance of $H(\cP_{W_i})$ is finite as a consequence of the assumption~\eqref{p-moment} by using the Poincar\'e inequality. Thus by the classical central limit theorem, for fixed $L > 0$, as $n \to \infty$,
\[
	X_{n, L} \dto \Normal(0, \sigma_L^2), \quad {\Var[X_{n, L}]} \to \sigma_L^2 = L^{-1}{\Var[H(\cP_{W_i})]},
\]
because $\ell_n / n \to 1/L$.

The sequence $X_{n, L}$ well approximates  $Y_n:= n^{-1/2}(H(\cP_n) - \Ex[H(\cP_n)] )$ in the following sense
\begin{equation}\label{variance-to-zero}
	\lim_{L \to \infty} \limsup_{n \to \infty}\Var[X_{n, L} - Y_n] = 0.
\end{equation}
Once this equation is proved, then a CLT for $Y_n$ follows from Lemma~\ref{lem:CLT-triangle}. Thus, the remaining task is to show~\eqref{variance-to-zero}. We note here that the idea of taking partition like this has been used to prove the strong law of large numbers for Betti numbers in the thermodynamic regime \cite{Trinh-2017, Yoge-2017}.

It follows from the Poincar\'e inequality  that 
\begin{align}
	&\Var[X_{n, L} - Y_n] \nonumber\\
	 &\le \frac\lambda n \int_{K_n} \Ex[|D_y(\cP_{n}) - \sum_i D_y(\cP_{W_i}) \one_{W_i}(y) |^2] dy \nonumber\\
	&= \frac{\lambda}{n} \int_{K_n \setminus (\cup_i W_i)} \Ex[|D_y(\cP_{n})|^2] dy + \frac \lambda n  \sum_i \int_{W_i} \Ex[|D_y(\cP_{n}) - D_y(\cP_{W_i})|^2] dy . \label{homo-Poincare}
\end{align}
Here we have used the Poincar\'e inequality for the functional 
\[
	H'(\fX) := H(\fX \cap K_n) - \sum_{i}H(\fX \cap W_i).
\]
The integrands in the above integrals are uniformly bounded by the assumption~\eqref{p-moment}, that is, there is a constant $C>0$ such that
\[
	\Ex[|D_y(\cP_{n})|^2] \le C, \quad \Ex[|D_y(\cP_{n}) - D_y(\cP_{W_i})|^2] \le C.
\]
Thus the first term in \eqref{homo-Poincare} vanishes as $n \to \infty$.

For the second term, note that the weak stabilization assumption, together with the uniform boundedness assumption~\eqref{p-moment}, implies that  
\[
	\Ex[|D_0(\cP_{V_n}) - \Delta|^2] \to 0, 
\]
for any sequence $\{V_n \ni 0\}$ of cubes tending to $\R^d$ as $n \to \infty$. It follows that for given $\varepsilon > 0$, we can choose a number $t > 0$  such that for any pair $(V,W)$ of cubes  with $B_t(0) \subset V\cap W$, 
\[
	\Ex[|D_0(\cP_{V}) - D_0(\cP_{W})|^2] < \varepsilon.
\]
Here $B_r(x)$ denotes the closed ball of radius $r$ centered at $x$ (with respect to the Euclidean metric). Note that the above inequality still holds if $0$ is replaced by any $y\in \R^d$ because of the translation invariance of $H$ and of $\cP(\lambda)$.

Let $\interior(W_i):=\{y \in W_i :  B_t(y) \subset W_i\}$ and $\partial(W_i): = W_i \setminus \interior(W_i)$, for $L > 2t$. Then $|\partial (W_i)| = L - (L^{1/d} - 2t)^d \le 2td L^{(d-1)/d}$. Here $|A|$ denotes the volume of a set $A$. Note that $W_i \subset K_n$. Thus for $y \in \interior(W_i)$, $\Ex[|D_y(\cP_{n}) - D_y(\cP_{W_i})|^2]<\varepsilon$. Then the second term in \eqref{homo-Poincare} can be estimated as follows (for $n, L > 2t$),
\begin{align*} 
	&\frac \lambda n  \sum_i \int_{W_i} \Ex[|D_y(\cP_{n}) - D_y(\cP_{W_i})|^2] dy \\
	&=\frac \lambda n  \sum_i \left (\int_{\interior(W_i)} \left(\cdots \right)dy  + \int_{\partial (W_i)} \left(\cdots \right)dy \right) \\
	&\le \frac \lambda n  \sum_i \left (\int_{\interior (W_i)} \varepsilon dy  + \int_{\partial (W_i)} C dy \right) \\ 
	&\le  \lambda \varepsilon + \frac{const}{L^{1/d}}.
\end{align*}
Therefore
\[
	\limsup_{L \to \infty} \limsup_{n \to \infty} 	\Var[X_{n, L} - Y_n]  \le \lambda \varepsilon,
\]
which implies the equation~\eqref{variance-to-zero} because $\varepsilon$ is arbitrary. The theorem is proved.
\end{proof}

\subsection{Non-homogeneous Poisson point processes}
Let $f \colon \R^d \to [0, \infty)$ be a bounded function with compact support. We are going to establish a central limit theorem for $H(n^{1/d} \cP(n f))$. When $f$ is a probability density function, then $\cP(n f)$ has the same distribution with the Poissonized version $\cP_n = \{X_1, \dots, X_{N_n}\}$. However, in this section, $f$ need not be a probability density function. Let $\tilde \cP_n = n^{1/d} \cP(n f)$. Then $\tilde \cP_n$ is a Poisson point process with intensity function $f(x/n^{1/d})$.

Let us discuss some terminologies. The functional $H$ is \emph{strongly stabilizing} on $\cP(\lambda)$ if there exist (finite) random variables $\tau(\lambda)$ (a radius of stabilization of $H$) and $\Delta(\lambda)$ (the limiting add one cost) such that almost surely, 
	\[
		D_0 ((\cP (\lambda)|_{B_{\tau(\lambda)}(0)})\cup \cA) = \Delta(\lambda),
	\]
for all finite $\cA \subset \R^d$ satisfying $\cA \cap B_{\tau(\lambda)}(0) = \emptyset$. It is clear that the strong stabilization implies the weak one.

The functional $H$ satisfies the \emph{Poisson bounded moments condition} on $\{\tilde \cP_n\}$ if there exists a constant $p > 2$ such that 
	\begin{equation}\label{bdd-moments}
		\sup_n \sup_{y \in \R^d} \sup_{y \in W\text{:cube}}\Ex[|D_y(\tilde \cP_n|_{W})|^{p}] < \infty.
	\end{equation}
We claim that this condition on $\{\tilde \cP_n\}$ implies the condition~\eqref{p-moment} on a homogeneous Poisson point process $\cP(\lambda)$ with density $\lambda = f(x)$, provided that $x$ is a Lebesgue point of $f$. Indeed, by definition, the point $x$ is a Lebesgue point of $f$ if 
\[
	\lim_{r \to 0+} \frac{1}{|B_r(x)|} \int_{B_r(x)} |f(y) - f(x)|dy = 0.
\]
Let $W \ni 0$ be a cube. Let $W_n = (n^{1/d}x + W)$ and $V_n = x + n^{-1/d}W$. Then $|V_n| = n^{-1}|W|$, and hence,
\[
	\int_{W_n} |f(y/n^{1/d}) - f(x)| dy = n \int_{V_n} |f(z) - f(x)| dz \to 0 \text{ as } n \to \infty.
\]
Lemma~\ref{lem:L1-convergence} applying to the shifted point process $(\tilde \cP_{n}|_{W_n} - n^{1/d} x)$ and to the add one cost function $D_0$ implies that 
\[
	D_{n^{1/d}x} (\tilde \cP_n|_{W_n}) = D_0(\tilde \cP_{n}|_{W_n} - n^{1/d} x) \dto D_0(\cP(\lambda)|_{W}).
\]
Then by Fatou's lemma, 
\[
	\Ex[|D_0(\cP(\lambda)|_{W})|^p] \le \limsup_{n \to \infty} \Ex[|D_{n^{1/d}x} (\tilde \cP_n|_{W_n})|^p],
\]
from which the condition~\eqref{p-moment} follows. Consequently, a CLT in Theorem~\ref{thm:homo} holds for $\lambda = f(x)$, where $x$ is a Lebesgue point  of $f$, under the assumption that $H$ is strongly stabilizing on $\cP(\lambda)$ and satisfies the Poisson bounded moments condition on $\{\tilde \cP_n\}$.

The functional $H$ satisfies the \emph{locally bounded moments condition} on $\{\tilde \cP_n\}$ if for any cube $W \subset \R^d$, there is a $\delta > 0$ such that
	\begin{equation}\label{locally}
		\sup_n \sup_{y} 	\Ex[|H(\tilde\cP_n|_{y+W})|^{2+\delta}] < \infty.
	\end{equation}
This condition is a technical one. So far, we do not know whenever or not it is a consequence of the  Poisson bounded moments condition.

Now we can state the main result in this section.
\begin{theorem}\label{thm:Poisson}
Let $f$ be a non-negative bounded function with compact support. Let $\Lambda = \sup f(x)$.
Assume that the functional $H$ is strongly stabilizing on $\cP(\lambda)$ for any $\lambda \in [0, \Lambda]$, satisfies the Poisson bounded moments condition and the locally bounded moments condition. 
Then as $n \to \infty$,
\[
	\frac{\Var[H(\tilde\cP_n)]}{n} \to \sigma^2,\quad 
	\frac{H(\tilde\cP_n) - \Ex[H(\tilde\cP_n )]}{\sqrt{n}} \dto \Normal(0, \sigma^2).
\]
Here $\sigma^2 = \int  \hat \sigma^2(f(x)) dx$, with $\hat \sigma^2(\lambda)$ being the limiting variance in Theorem~{\rm\ref{thm:homo}}. 
\end{theorem}

\begin{remark}
From the argument following the Poisson bounded moments condition, we see that the limiting variance $\hat \sigma^2 (f(x))$ is defined at every Lebesgue point $x$ of $f$. Moreover, the Lebesgue differentiation theorem states that for an integrable function $f$, almost every point is a Lebesgue point. Thus $ \hat \sigma^2 (f(x))$ is defined almost everywhere.
\end{remark}

We use the same idea as in  the proof of Theorem~\ref{thm:homo}. Let $S$ be a cube which contains the support of $f$. For $L>0$, divide $\R^d$ according to the lattice $(L/n)^{1/d} \Z^d$ and let $\{V_i\}$ be the cubes which intersect with $S$. Set $S_n = \cup_i V_i$. Then the number of $\{V_i\}$, denoted by $\ell_n$, satisfies that $\ell_n/n =|S_n|/L \to |S|/L$ as $n \to \infty$. Let $W_i$ be the image of $V_i$ under the map $x \mapsto n^{1/d}x$. Recall that $\tilde \cP_n$ is a Poisson point process on $\tilde S_n = n^{1/d}S$ with intensity function $f(x/n^{1/d})$. Assume that the functional $H$ satisfies all the assumptions in Theorem~\ref{thm:Poisson}.

Let 
\[
	X_{n, L} = \frac{1}{\sqrt{n}} \sum_{i} \Big( H(\tilde \cP_n|_{W_i})-\Ex[H(\tilde \cP_n|_{W_i})]  \Big).
\]
\begin{lemma}\label{lem:variance-bdd}
There is a constant $M>0$ such that for any cube $W$,
	\[
		\frac{\Var[H(\tilde \cP_n|_{W})]}{|W|} \le M.
	\]
\end{lemma}
\begin{proof}
It follows from the Poisson bounded moments condition that there is a constant $C>0$ such that for all $n$, all $y$ and all $W \ni y$,
	\begin{equation}\label{2nd-bdd}
		\Ex[|D_y(\tilde \cP_n|_{W})|^{2}] \le 		\Ex[|D_y(\tilde \cP_n|_{W})|^{p}]^{2/p} \le C.
	\end{equation}
Then the desired estimate is just a consequence of the Poincar\'e inequality 
	\[
		\Var[H(\tilde \cP_n|_{W})]   \le \int_{W} \Ex[|D_y(\tilde \cP_n|_{W})|^2] f(y/n^{1/d}) dy \le C \Lambda |W|. \qedhere
	\]
\end{proof}

\begin{lemma}\label{lem:Lebesgue-point}
Let $x$ be a Lebesgue point of $f$. For each $n$, let $V_{n}$ be a cube of volume $L/n$ containing $x$. Let $W_n = n^{1/d} V_n$.
Then as $n \to \infty$,
\[
	\Ex[H(\tilde \cP_n|_{W_n})] \to \Ex[H(\cP_L(\lambda))], \quad	\Var[H(\tilde \cP_n|_{W_n})] \to \Var[H(\cP_L(\lambda))],
\]
where $\lambda = f(x)$ and $\cP_L(\lambda)$ denotes the restriction of $\cP(\lambda)$ on a cube of volume $L$. In particular, we also have $L^{-1}\Var[H(\cP_L(\lambda))] \le M$, where $M$ is the constant in Lemma~{\rm\ref{lem:variance-bdd}}.
\end{lemma}
\begin{proof}
Recall from the derivation of the condition~\eqref{p-moment}  from the condition~\eqref{bdd-moments} that
\[
	\int_{W_n} |f(y/n^{1/d}) - f(x) | dy \to 0 \text{ as } n \to \infty.
\]
Together with the locally bounded moments condition, all the conditions in Corollary~\ref{cor:L1-convergence} are satisfied. Thus the convergences of expectations and variances follow. The proof is complete.
\end{proof}

\begin{lemma}\label{lem:CLT-triangle-2}
	For fixed $L > 0$, as $n \to \infty$, 
	\[
		{\Var[X_{n, L}]} \to  \int_S \frac{\Var[H(\cP_L(f(x)))]}{L}  dx=: \sigma_L^2,\quad  X_{n, L} \dto \Normal(0, \sigma_L^2).
	\]
\end{lemma}
\begin{proof}
Let us first show the convergence of variances. 
We write the variance of $X_{n, L}$ as follows
\begin{align*}
	{\Var[X_{n, L}]} &= \frac{1}{n}  \sum_{i} \Var[H(\tilde \cP_n|_{W_i})] \\
	&= \sum_{i} \frac{\Var[H(\tilde \cP_n|_{W_i}))]}{L}  \frac{L}{n}\\
	&=\int_{S_n}  \sum_{i} \frac{\Var[H(\tilde \cP_n|_{W_i}))]}{L} \one_{V_i}(x) dx\\
	&=:\int_{S_n} g_{n, L}(x) dx.
\end{align*}
It follows from Lemma~\ref{lem:variance-bdd} that $|g_{n, L}(x)| \le M$. Moreover, when $x \in S$ is a Lebesgue point of $f$, then by Lemma~\ref{lem:Lebesgue-point}, as $n \to \infty$,
\[
	g_{n, L}(x)= \frac{\Var[H(\tilde \cP_n|_{W_{i(x,n)}})]}{L}\to \frac{\Var[H(\cP_L(f(x)))]}{L}.
\]
Here $V_{i(x,n)} = n^{-1/d}W_{i(x,n)}$ is the unique cube in $\{V_i\}$ containing $x$.
In addition, it is clear that $|S_n \setminus S| \to 0$ as $n \to \infty$. Recall that almost every $x \in S$ is a Lebesgue point.  Therefore the convergence of the variance $\Var[X_{n, L}]$ follows by the bounded convergence theorem.

The CLT for $X_{n, L}$ then follows from Theorem~\ref{thm:CLT-Lyapunov} because the locally bounded moments condition has been assumed.  The proof is complete.
\end{proof}

\begin{lemma}\label{lem:good-approximation}
The following holds
\[
	\lim_{L \to \infty} \limsup_{n \to \infty} \Var\bigg[\frac{H(\tilde \cP_n) - \Ex[H(\tilde \cP_n) ]}{\sqrt{n}} - X_{n, L} \bigg] = 0.
\]
\end{lemma}
\begin{proof}
We begin with the Poicar\'e inequality 
\begin{align*}
	&\Var\bigg[\frac{H(\tilde \cP_n) - \Ex[H(\tilde \cP_n) ]}{\sqrt{n}} - X_{n, L} \bigg]=\frac{\Var[H(\tilde \cP_n) - \sum_i H(\tilde \cP_n|_{W_i})]}{n} \\
	 &\le \frac1n \int_{\tilde S_n} \Ex[|D_y(\tilde \cP_n) - \sum_i D_y(\tilde \cP_n|_{W_i}) \one_{W_i}(y) |^2] f(y/n^{1/d}) dy\\
	&= \frac1n \sum_i \int_{W_i} \Ex[|D_y(\tilde \cP_n) - D_y(\tilde \cP_n|_{W_i})|^2] f(y/n^{1/d}) dy.
\end{align*}
It follows from~\eqref{2nd-bdd} that, 
\[
	 \Ex[|D_y(\tilde \cP_n) - D_y(\tilde \cP_n|_{W_i})|^2] \le 4 C.
\]
Let $t>0$. Assume that $L >2t$. Recall the notations $\interior( W_i)$ and $\partial( W_i)$ from the proof in the homogeneous case.
Then 
\[
	\int_{\partial W_i} \Ex[|D_y(\tilde \cP_n) - D_y(\tilde \cP_n|_{W_i})|^2] f(y/n^{1/d}) dy
	\le {4 C \Lambda  2td L^{(d-1)/d}},
\]
and hence,
\begin{equation}\label{boundary-estimate}
	\limsup_{n \to \infty} \frac1n \sum_i \int_{\partial W_i}\Ex[|D_y(\tilde \cP_n) - D_y(\tilde \cP_n|_{W_i})|^2] f(y/n^{1/d}) dy  \le \frac{const}{L^{1/d}}.
\end{equation}

Next, we deal with the case $y \in {\rm int\,}( W_i)$. Let $x = y/n^{1/d}$.
Consider a homogeneous Poisson point process $\cP(\lambda)$ with density $\lambda = f(x)$. Let $\tau(\lambda)$ be the stabilization radius of $H$ on $\cP(\lambda)$ at $y$. There is a coupling of $\cP(\lambda)$ and $\tilde \cP_n$ such that (see the proof of Lemma~\ref{lem:L1-convergence})
\[
	\Prob(A=\{\tilde \cP_n|_{W_i} \equiv \cP(\lambda)|_{W_i}\})  = e^{-\tilde t_n(y)},
\] 
where $\tilde t_n(y) =t_n(x)= \int_{W_i} |f(y/n^{1/d}) - f(z/n^{1/d})|dz = n\int_{V_i} |f(x) - f(z)|dz$. On the event $A \cap \{\tau(\lambda) \le t\}$, by the definition of the radius of stabilization, $D_y(\tilde \cP_n) = D_y(\tilde \cP_n|_{W_i})$. Thus 
\begin{align*}
	&\Ex[|D_y(\tilde \cP_n) - D_y(\tilde \cP_n|_{W_i})|^2] \\
	&= \Ex[|D_y(\tilde \cP_n) - D_y(\tilde \cP_n|_{W_i})|^2; A^c  \cup \{\tau(\lambda) > t\}]\\
	&\le \Ex[|D_y(\tilde \cP_n) - D_y(\tilde \cP_n|_{W_i})|^{p}]^{2/p}\Prob(A^c  \cup \{\tau(\lambda) > t\})^{1/q} \\
	&\le C_p (1- e^{-\tilde t_n(y)} + \Prob(\tau(\lambda) > t))^{1/q}.
\end{align*}
Here $C_p$ is a constant which comes from the Poisson bounded moments condition~\eqref{bdd-moments}. We have used H\"older's inequality with $q$ being the H\"older conjugate number of $p/2$.
Therefore 
\begin{align*}
	&\int_{{\rm int\,}(W_i)} \Ex[|D_y(\tilde \cP_n) - D_y(\tilde \cP_n|_{W_i})|^2] f(y/n^{1/d}) dy \\
	&\le \int_{{\rm int\,}(W_i)} C_p(1- e^{-\tilde t_n(y)} + \Prob(\tau(\lambda) > t))^{1/q} f(y/n^{1/d})dy \\
	&\le n \int_{V_i} C_p (1- e^{-t_n(x)} + \Prob(\tau(f(x)) > t))^{1/q} f(x)dx. 
\end{align*}
Note that $t_n (x) = n \int_{V_i} |f(x) - f(z)| dz \to 0$ as $n \to \infty$, for a Lebesgue point $x$ of $f$. Therefore, 
\begin{align}
	&\limsup_{n \to \infty} \frac{1}{n}\sum_i \int_{{\rm int\,}(W_i)} \Ex[|D_y(\tilde \cP_n) - D_y(\tilde \cP_n|_{W_i})|^2] f(y/n^{1/d}) dy \nonumber\\
	&\le C_p \limsup_{n \to \infty}  \int_{S}  (1- e^{-t_n(x)} + \Prob(\tau(f(x)) > t))^{1/q} f(x)dx \nonumber\\
	&\le C_p\int_S  \Prob(\tau(f(x)) > t)^{1/q} f(x)dx. \label{interior-estimate}
\end{align}
Here the bounded convergence theorem has been used in the last estimate. 
Combining the two estimates \eqref{boundary-estimate} and \eqref{interior-estimate}, we arrive at
\begin{align*}
	&\limsup_{L \to \infty}\limsup_{n \to \infty}\frac{\Var[H(\tilde \cP_n) - \sum_i H(\tilde \cP_n|_{W_i})]}{n}  \\
	&\le  C_p \int_S   \Prob(\tau(f(x)) > t)^{1/q} f(x)dx.
\end{align*} 
The proof is complete by letting $t \to \infty$.
\end{proof}

Similar to the homogeneous case, a central limit theorem for $H(\tilde \cP_n)$ follows by combining Lemma~\ref{lem:CLT-triangle-2} and Lemma~\ref{lem:good-approximation}. For the limiting variance, recall that when $x$ is the Lebesgue point of $f$, 
\[
	\frac{\Var[H(\cP_L(  \lambda) )]}{L}  \to \hat\sigma^2(\lambda) \text{ as $L \to \infty$}, \quad \lambda = f(x).
\] 
Recall also from Lemma~\ref{lem:Lebesgue-point} that $
	L^{-1}{\Var[H(\cP_L(  \lambda) )]} \le M.
$
Thus, by the bounded convergence theorem again, it follows that as $L \to \infty$,
\[
	\sigma_L^2  =  \int_S \frac{\Var[H(\cP_L(f(x)))]}{L}  dx \to \int_S \hat \sigma^2(f(x)) dx = \sigma^2.
\]
Theorem~\ref{thm:Poisson} is proved.

\subsection{Binomial point processes}
Let $\{X_i\}_{i = 1}^\infty$ be an i.i.d.~sequence of $\R^d$-valued random variables with a common probability density function $f$. The function $f$ is assumed to be bounded and to have compact support. Let $\fX_n = \{X_1, \dots, X_n\}$ and $\cP_n = \{X_1, \dots, X_{N_n}\}$ be the binomial point processes and the Poisson point processes associated with $\{X_i\}$, respectively. Assume that the functional $H$ satisfies all the assumptions of Theorem~\ref{thm:Poisson}. Then a CLT for $H(n^{1/d} \cP_n)$ holds, that is, as $n \to \infty$,
\begin{align*}
	&\frac{\Var[H(n^{1/d}\cP_n)]}{n} \to  \int \hat \sigma^2(f(x))dx =: \sigma^2,\\
	&\frac{H(n^{1/d}\cP_n)  - \Ex[H(n^{1/d}\cP_n)]}{\sqrt{n}} \dto \Normal(0, \sigma^2).
\end{align*}
Here recall that $\hat\sigma^2(\lambda) = \lim_{n \to \infty} n^{-1}\Var[H(\cP(\lambda)|_{[1, n^{1/d})^d})]$ is the limiting variance in the homogeneous case.

We now use a de-Poissonization technique to derive a CLT for $H(n^{1/d} \fX_n)$. It turns out that we only need two more moments conditions. 
The first one requires that there is a constant $\beta > 0$ such that for any $m,n$
\[
	H(n^{1/d} \fX_m) \le \beta (m + n)^\beta, \text{almost surely}.
\] 
The second one requires  
\[
	\sup_{n \in \N} \sup_{m \in [(1-\eta)n, (1+\eta)n]} \Ex[|H(n^{1/d}\fX_{m + 1}) - H(n^{1/d}\fX_m)|^q] < \infty,
\]
for some $q> 2, \eta > 0$. When the two conditions are added, the following CLT for $H(n^{1/d}\fX_n)$ holds
\begin{align*}
	\frac{H(n^{1/d}\fX_n) - \Ex[H(n^{1/d}\fX_n)]}{\sqrt{n}} \dto \Normal(0, \tau^2) \text{ as } n \to \infty,
\end{align*}
where $\tau^2 = \sigma^2 - (  \int \Ex[\Delta(f(x))] f(x) dx)^2 \ge 0$. The convergence of variances also holds. Recall that $\Delta(\lambda)$ is the limiting add one cost on a homogeneous Poisson point process $\cP(\lambda)$. Moreover, if the distribution of $\Delta(\lambda)$ is nondegenerate for $\lambda \in A$, where $\Prob(X_1 \in A) > 0$, then $\tau^2 > 0$ and $\sigma^2 > 0$. 

The above derivation is just a direct application of Theorem~2.16 in \cite{Penrose-book}. A detailed discussion on this de-Poissonization technique can be found in Section 2.5 of the same book.

\section{CLT for Betti numbers}

For a finite set of points $\fX=\{x_1, \dots, x_n\}$ in $\R^d$, the \v{C}ech complex of radius $r>0$, denoted by $\cC(\fX, r)$, is defined as an abstract simplicial complex consisting of non-empty subsets of $\fX$ in the following way
\[
	\{x_{i_0}, \dots, x_{i_k}\} \in \cC(\fX, r) \Leftrightarrow \bigcap_{j = 0}^k B_r(x_{i_j}) \neq \emptyset.
\] 
The nerve theorem tells us that the abstract simplical complex $\cC(\fX, r)$ is homotopy equivalent to the union of balls 
\[
	U_r (\fX) = \bigcup_{i = 1}^n B_r(x_i).
\]
\v{C}ech complexes may be regarded as a generalization of geometric graphs. 

Denote by $\beta_k(\cC(\fX, r))$ the $k$th Betti number, or the rank of the $k$th homology group of $\cC(\fX, r)$, with coefficients from some underlying field. The limiting behavior of $\beta_k(\cC(\fX_n, r_n))$ has been study intensively, where $\{r_n\}$ is a deterministic sequence tending to zero. It is known that Betti numbers behave differently in three regimes divided according to the limit of $\{n^{1/d}r_n\}$: zero, finite, or infinite. Refer to a survey \cite{BK} for more details on this topic. Note that the zeroth Betti number $\beta_0(\cC(\fX, r))$ just counts the number of connected components in $U_r(\fX)$. Also $\beta_k(\cC(\fX, r)) = 0$, if $k \ge d$, as a consequence of the nerve theorem. 

We focus now on the thermodynamic regime, also called the critical regime, in which $n^{1/d}r_n \to r \in (0, \infty)$. Without loss of generality, we may assume that $n^{1/d}r_n = r$. Define a functional $H_r$ as 
\[
	H_r(\fX) = \beta_k(\cC(\fX, r)).
\] 
Then it is clear that in this regime $\beta_k(\cC(\fX_n, r_n)) = H_r(n^{1/d} \fX_n)$, which is exactly the scaling considered in this paper. 

The following results in the thermodynamical regime have been known.

(i) \emph{Homogeneous Poisson point processes.} The following strong law of large numbers (SLLN) and CLT hold \cite{Yoge-2017}. For $0\le k \le d-1$, as $n \to \infty$, 
	\begin{align*}
		&\frac{\beta_k(\cC(\cP(\lambda)|_{[0, n^{1/d})^d}, r))}{n} \to \hat \beta_k (\lambda, r), \text{almost surely,}\\
		&\frac{\beta_k(\cC(\cP(\lambda)|_{[0, n^{1/d})^d}, r)) - \Ex[\beta_k(\cC(\cP(\lambda)|_{[0, n^{1/d})^d}, r))]}{\sqrt n} \dto \Normal(0, \hat\sigma^2_k(\lambda, r)).
	\end{align*}
	Here $\hat \beta_k (\lambda, r)$ and $\hat\sigma^2_k(\lambda, r)$ are constants, $\hat \beta_k (\lambda, r) > 0$ and  $\hat\sigma^2_k(\lambda, r) > 0$, for $\lambda, r>0$. Note that the CLT follows from a general result in \cite{PY-2001} by showing that $H_r$ is weakly stabilizing on $\cP(\lambda)$. (Moments conditions for Betti numbers can be verified relatively easily.)
	 These results on Betti numbers are generalized to persistent Betti numbers in \cite{HST-2017}.

(ii) \emph{Binomial point processes.} The following SLLN holds. Assume that the probability density function $f$ is bounded and has compact support. Then as $n \to \infty$,
	\[
		\frac{\beta_k(\cC(n^{1/d} \fX_n, r))}{n} \to \int \hat \beta_k (f(x), r)dx, \text{almost surely.}
	\]
	A partial of this result in which some additional conditions on $f$ are required is a combination of \cite[Theorem~4.6]{Yoge-2017} and \cite[Theorem~1.3]{Trinh-2017}. In a forthcoming work \cite{Goel-2018}, we are able to remove such technical conditions.

	It is clear that $H_r$ is strongly stabilizing if almost surely, $U_r(\cP(\lambda))$ does not have infinite connected component because Betti numbers are additive on connected components. Let $r_c = r_c (d)$ be the critical radius for percolation of the occupied component
	\[
		r_c = \inf \{r : \Prob(U_r(\cP(1)) \text{ has infinite connected component}) > 0 \}.
	\]
It is known from the theory of continuum percolation theory that $0 < r_c < \infty$ \cite{Meester-Roy-book}. Thus for $r<r_c$, almost surely, $U_r(\cP(1))$ does not have infinite component. This implies the strong stabilization of $H_r$ on $\cP(1)$ when $r < r_c$.
By a scaling property of homogeneous Poisson point processes  ($\cP(\lambda)$ has the same distribution with $\lambda^{-1/d}\cP(1)$), it follows that $H_r$ is strongly stabilizing on $\cP(\lambda)$, if $r < \lambda^{-1/d} r_c$. Therefore, the following CLT for Betti numbers is an application of our general result. 
\begin{theorem} Let $f$ be a bounded probability density function with compact support. Let $\Lambda = \sup f(x)$. Then for $0 \le k \le d-1$, 
as $n \to \infty$ with $n^{1/d} r_n \to r \in (0, \Lambda^{-1/d} r_c)$,
\begin{align*}
	\frac{\beta_k(\cC(\cP_n, r_n)) - \Ex[\beta_k(\cC(\cP_n , r_n))]}{\sqrt n} &\dto \Normal(0, \sigma^2_k) , \quad  \sigma^2_k = \int \hat\sigma^2_k(f(x), r)dx, \\
	\frac{\beta_k(\cC(\fX_n, r_n)) - \Ex[\beta_k(\cC(\fX_n , r_n))]}{\sqrt n} &\dto \Normal(0, \tau^2_k), \quad \sigma^2_k > \tau^2_k > 0.
\end{align*}
\end{theorem}

Note that $\beta_0$ is strongly stabilizing without any restriction on $r$ because the infinite component, when exists, is unique. As a consequence of the general result here, Theorem~13.26 and Theorem~13.27 in \cite{Penrose-book} still hold without the Riemann integrable assumption on $f$.

Note also that by a duality property, it was shown in \cite{Yoge-2017} that $H_r$ is strongly  stabilizing  on $\cP(1)$, if $r \notin I_d$,
where 
\[
	I_d = \begin{cases}
		(r_c, r_c^*],&\text{if } \Prob(U_{r_c}(\cP(1)) \text{ has infinite connected component}) = 0,\\
		[r_c, r_c^*], &\text{otherwise,}
	\end{cases}
\]
$r_c^*$ being the critical radius for percolation of the vacant component
\[
	r_c^* = \sup\{r : \Prob(\R^d \setminus U_{r}(\cP(1)) \text{ has infinite connected component}) > 0\}.
\]
In particular, $I_2 = \emptyset$, which implies that for $d=2$, $H_r$ is strongly stabilizing on $\cP(\lambda)$ for all $\lambda$. Thus in two dimensional case, there is no restriction on $r$.

\subsection*{Acknowledgment} 
The author would like to thank Professor Tomoyuki Shirai and Dr.~Kenkichi Tsunoda for many useful discussions.
This work is partially supported by JST CREST
Mathematics (15656429) and JSPS KAKENHI Grant Numbers JP16K17616.

\begin{footnotesize}
\bibliographystyle{spmpsci}
\bibliography{bpp-a}
\end{footnotesize}

\end{document}